\documentclass[11 pt]{amsart}
\usepackage{latexsym,amsmath,amsfonts,graphicx}
\usepackage{subfigure}
\usepackage{multicol}
\usepackage{amssymb}
\usepackage{float}
\usepackage[dvips]{color}
\newtheorem{theorem}{Theorem}[section]
\newtheorem{thm}[theorem]{Theorem}
\newtheorem{pro}{Proposition}[section] 

\newtheorem{lemma}[pro]{Lemma}
\newtheorem{remark}[pro]{Remark}
\newtheorem{defi}[pro]{Definition}

\newtheorem{example}{Example}[section]
\numberwithin{equation}{section}

\def\cal{\mathcal }

\def\N{\mathbb N}
\def\Z{\mathbb Z}

\def\mathscr{\mathcal }

\newcommand{\balpha}{{\boldsymbol{\alpha}}}

\newcommand{\bzero}{{\boldsymbol{0}}}
\newcommand{\bd}{{\boldsymbol{d}}}

\newcommand{\bdelta}{{\boldsymbol{\delta}}}

\newcommand{\bm}{{\mathbf{m}}}

\newcommand{\be}{{\mathbf e}}

\newcommand{\ba}{{\boldsymbol{a}}}

\newcommand{\bb}{{\mathbf{b}}}
\newcommand{\bc}{{\mathbf{c}}}
\newcommand{\bp}{{\mathbf{p}}}
\newcommand{\bq}{{\mathbf{q}}}

\newcommand{\SJ}{{\mathcal J}}

\begin{document}

\title[Complementing pairs of $(\Z_{\geq 0})^n$]{Characterization of complementing pairs of $({\mathbb Z}_{\geq 0})^n$}

\author{Hui Rao}
\address{School of Mathematics and  Statistic, Huazhong University of Science and Technology, Wuhan, 430074, China}
\email{hrao@mail.ccnu.edu.cn}

\author{Ya-min Yang}
\address{Institute of applied mathematics, College of Science, Huazhong Agriculture University, Wuhan,430070, China.}
\email{yangym09@mail.hzau.edu.cn}

\author{Yuan Zhang$^*$}
\address{School of Mathematics and  Statistic, Huazhong University of Science and Technology, Wuhan, 430074, China}
\email{yzhangccnu@sina.cn}

\date{\today}
\thanks {The work is supported by NSFS Nos. 11971195 and 11601172.}

\thanks{* The correspondence author.}
\begin{abstract} Let $A, B, C$ be subsets of an abelian group $G$.
A pair $(A, B)$ is called a $C$-pair if $A, B\subset C$ and $C$ is the direct sum of $A$ and $B$.
 The $(\Z_{\geq 0})$-pairs are characterized by  de Bruijn in 1950 and the $(\Z_{\geq 0})^2$-pairs are characterized
 by Niven in 1971. In this paper, we characterize the $(\Z_{\geq 0})^n$-pairs for all $n\geq 1$. We show that
every $(\Z_{\geq 0})^n$-pair is characterized by a weighted tree if it is primitive, that is, it is not a Cartesian product of a $(\Z_{\geq 0})^p$-pair and a $(\Z_{\geq 0})^q$-pair of lower dimensions.
\end{abstract}

\maketitle


\section{\textbf{Introduction}}\label{sec:intro}
Let $\Z$ be the set of integers. Let $A, B, C\subset \Z^n$.  We define
$A+B=\{a+b;~a\in A, b\in B\}.$
We  call $A+B$  the \emph{direct sum} of $A$ and $B$ and denoted by $A\oplus B$, if
 every element $c\in A+B$ has a unique decomposition as $c=a+b$ with $a\in A, b\in B$.
We say $(A, B)$ is a \emph{complementing pair} of $C$, or a \emph{$C$-pair},  if $C=A\oplus B$ and $A, B\subset C$.

We are interested in $\Z^n$-pairs and $\N^n$-pairs, where we denote $\N=\Z_{\geq 0}$
for simplicity.
Furthermore,   a $\Z^n$-pair $(A, B)$ is called a \emph{$\Z^n$-tiling} if   $\#A<\infty$, where $\#A$ denotes the cardinality of $A$; in this case,
we call $A$ a \emph{$\Z^n$-tile}. Similarly, we define $\N^n$-tiling and $\N^n$-tile.

The characterization of   $\Z$-pairs  is an important problem in additive number theory.
 This problem  is first considered by de Bruijn \cite{deB50} in 1950.
In 1974, Swenson \cite{Swenson}  showed that this problem is   an \emph{NP}-hard problem.
There are very few results on this problem.

To determine when a finite set $A\subset \Z$ is a   $\Z$-tile is  somehow easier.
This was done by Newman \cite{Newman}  when $\#A$ is a power of a prime number, see also Tijdeman \cite{Tij}.
 Coven and Meyerowitz \cite{CM99} solved the problem
 when $\#A$ contains at most two prime factors, see also Sands \cite{Sands}.
  If $\#A$ contains more than
  two prime factors, the problem is still widely open \cite{Szabo, CM99}.
There are almost no   results on $\Z^n$-pairs with $n\geq 2$,  except
that  Szegedy \cite{Szegedy}  characterized the $\Z^d$-tile $T$
  in case that $\#T$   is a prime number or $\#T=4$.

The characterization of $\N$-pairs is much easier and  it was settled by de Bruijn \cite{deB56}, and was rediscovered by
Vaidya \cite{Vaidya}.
Let $(n_k)_{k\geq 1}$ be a sequence of integers with $n_k\geq 2$.  Any $t\in \N$ can be written uniquely in the form
\begin{equation}\label{eq:baseN}
t=  d_1+n_1d_2+(n_1n_2)d_3+\cdots+(n_1\cdots n_{L-1}) d_{L},
\end{equation}
where $d_j\in \{0,1,\dots, n_{j}-1\}$ for each $1\leq j\leq L$ and $d_L>0$. We denote the right hand side of
 \eqref{eq:baseN} by  $\overline{d_1\dots d_L}$.

\begin{pro}\label{thm:deB} (de Bruijn \cite{deB56}, Vaidya \cite{Vaidya})
(i)  A pair $(T,\SJ)$ is an $\N$-pair with $\#T=\#\SJ=\infty$ if and only if there exists a sequence of integers $(n_k)_{k\geq 1}$ with $n_k\geq 2$
such that
\begin{equation}\label{eq:base}
T=\bigcup_{L=1}^\infty \{\overline{d_1\dots d_L};~ d_j=0 \text{ if $j$ is odd}\}, \
\SJ=\bigcup_{L=1}^\infty \{\overline{d_1\dots d_L};~ d_j=0 \text{ if $j$ is even}\},
\end{equation}
or the other round.

(ii) $T$ is an $\N$-tile if and only if there exist $L\geq 2$ and  $(n_k)_{k= 1}^{L}$ with $n_k\geq 2$
such that
$$
T= \{\overline{d_1\dots d_L};~ d_j=0 \text{ if $j$ is odd}\} \text{ or } T= \{\overline{d_1\dots d_L};~ d_j=0 \text{ if $j$ is even}\}.
$$
\end{pro}

\begin{remark}{\rm Let $(T,\SJ)$ be an $\N$-pair with $\#T=\#\SJ=\infty$.
Long \cite{Long} made the interesting observation that $(T,-\SJ)$ is a $\Z$-pair;  Eigen and Hajian \cite{Eigen} showed that there exists a continuum number of sets $\widetilde \SJ$ such that $(T,\widetilde \SJ)$ is a $\Z$-pair.
}
\end{remark}

After a partial result was obtained
by Hansen \cite{Hansen},
Niven \cite{Niven} characterized the $\N^2$-pairs  (See Example \ref{exam:Niven} in the end of this section).

The  goal of the present paper is to characterize   the $\N^n$-pairs for all $n\geq 1$.
For simplicity, we call $(T, \SJ)$ an \emph{$\N^*$-pair}
if   $(T, \SJ)$ is an $\N^n$-pair for some $n\geq 1$.

It is easy to show that   if $(I_1,J_1)$ is an $\N^n$-pair and $(I_2, J_2)$ is an $\N^m$-pair, then
$(T,\SJ)=(I_1\times I_2,  J_1\times J_2)$ is an $\N^{n+m}$-pair, and we call $(T,\SJ)$ the Cartesian product of the two pairs.  (See Lemma \ref{lem:unique}.)

An $\N^n$-pair $(T,\SJ)$ is called \emph{primitive}, if it is not a Cartesian product of
two $\N^*$-pairs with lower dimensions. All $\N$-pairs are primitive.
Clearly, to characterize  $\N^*$-pairs, we only need  understand   primitive $\N^*$-pairs.

In literature, polynomials have been used to study direct sum decomposition of abelian groups (\cite{Haj, Sands, CM99}) or finite sets (\cite{Nath72}). The first  idea in this paper is to use power series  to handle $\N^*$-pairs. A power series with finite many variables is called a \emph{binary power series} if its coefficients belong to $\{0,1\}$.
Let $x=(x_1,\dots, x_n)$ be an $n$-tuple variable. For $\ba=(a_1,\dots, a_n)\in \N^n$,
we denote    $x^\ba=x_1^{a_1}\cdots x_n^{a_n}$.
For $A\subset \N^n$, we define
\begin{equation}\label{eq:mask}
A(x)=\sum_{\ba\in A} x^\ba.
\end{equation}
For example $\N(x)=\sum_{k\geq 0}  x^k$ and $\N^2(x,y)=\N(x)\N(y)$.
Clearly, $(T,\SJ)$ is a   $\N^n$-pair if and only if
$T(x)\SJ(x)=\N^n(x).$

We call $(C, D)$ an \emph{interval pair} if
it is a $\{0,1,\dots, N-1\}$-pair for some integer  $N\geq 2$, and we call $N$ the \emph{size} of the pair $(C,D)$.

\begin{remark}{\rm Nathanson \cite{Nath72} proved that if
$T\oplus S=\prod_{j=1}^n\{0,1,\dots, a_j\}$
is a higher dimensional interval pair,
then $(T,S)$ is a cartesian product of one-dimensional interval pairs.
}
\end{remark}

\subsection{Extension process}

Now, we introduce two ways to produce $\N^*$-pairs from a known $\N^n$-pair $(T, \SJ)$.
Let $j\in \{1,\dots, n\}$.
Denote   $x=(\bar x, x_j, \bar{\bar x})$, where
$ \bar x=(x_1,\dots, x_{j-1})$ and  $\bar{\bar x}=(x_{j+1}, \dots, x_n).$

\noindent \textit{(1) Extension of the first type.}
 Let $N\geq 2$   and $(C,D)$
be a $\{0,1, \dots, N-1\}$-pair.  Set
\begin{equation}\label{eq:ext-1}
T^*(x)=C(x_j)T(\bar x, x_j^N,\bar{\bar x}), \quad
\SJ^*(x)=D(x_j)\SJ(\bar x, x_j^N,\bar{\bar x}).
\end{equation}
Then $(T^*,\SJ^*)$ is an $\N^n$-pair, and we call it a first type extension of $(T, \SJ)$.
(See Lemma \ref{lem:first-type}.)

\noindent \textit{(2) Extension of the second type.}
Let $\delta\in \{0,1\}$,
$m\geq 2$,  and  $\ba\in (\N\setminus\{0\})^m$. Set
$$
L_\ba=\N^m\setminus(\N^m+\ba).
$$
Denote $y=(y_1,\dots, y_m)$, and define $L_\ba(y)$ as we did in \eqref{eq:mask}. (For example, if $\ba=(1,1)$, then $L_\ba(y_1,y_2)=1+\sum_{k\geq 1}y_1^k+\sum_{k\geq 1}y_2^k$.)
Set
\begin{equation}\label{eq:ext-2}
T^*(\bar x, y, \bar{\bar x})=L^\delta _{\ba}(y)T(\bar x, y^\ba, \bar{\bar x}), \quad
\SJ^*(\bar x, y, \bar{\bar x})=L^{1-\delta}_{\ba}(y) \SJ(\bar x, y^\ba, \bar{\bar x}).
\end{equation}
Then $(T^*, \SJ^*)$ is an ${\N}^{n+m-1}$-pair, and we call it a second type extension of $(T, \SJ)$.
(See Lemma \ref{lem:second-type}.)

\medskip

Let $(T_j,\SJ_j)_{j=0}^k$ be a  sequence of $\N^{*}$-pairs. If
$(T_j, \SJ_j)$ is a first type or second type extension of $(T_{j-1}, \SJ_{j-1})$ for  $j=1,\dots, k$,
then we say $(T_k, \SJ_k)$ is a \emph{finite extension} of $(T_0, \SJ_0)$.

An  extension of an  $\N^*$-pair $(T, \SJ)$ is called an \emph{illegal extension}, if
 it is  a second type extension in \eqref{eq:ext-2}
satisfying $(T, \SJ)=(\{0\}, \N)$ and $\delta=0$, or $(T, \SJ)=(\N, \{0\})$ and $\delta=1$.
 An extension changes the primitivity if and only if it is illegal (Theorem \ref{thm:separable}).
Our first main result is the following.

\begin{thm}\label{thm:main} An $\N^*$-pair
$(T, \SJ)$ is  primitive  if and only if it is a finite extension of an $\N$-pair,
 and the extension process contains no illegal extension.
\end{thm}

 A finite extension of an $\N$-pair is  primitive   is proved in Section 3. The other direction of
  Theorem \ref{thm:main} is proved in Section 4-8, where
  Lemma \ref{lem:high-dim} and Theorem \ref{thm:marginal} contain  new techniques and  new phenomena which are different from dimension one and two.

\subsection{Weighted tree of an $\N^*$-pair.}
For a primitive $\N^*$-pair, we introduce  a weighted tree to record the information of the extension process
in Theorem \ref{thm:main}. Indeed, a weighted tree  provides a visualization of the structure of the associated  $\N^*$-pair.

 Let $(V, \Gamma)$ be a  finite tree with a root $\phi$, where $V$ is the node set and $\Gamma$ is the edge set.
  A node is called a \emph{top} if it has no offspring.
 Let $V_0$ denote the set of tops of $(V, \Gamma)$.
 In Section 9, we define a \emph{weight} of $(V,\Gamma)$ to be a quadruple
 \begin{equation}\label{w-tree}
 \left((T_0,\SJ_0),\{\bdelta_v\}_{v\in V\setminus V_0},\{\balpha_v\}_{v\in V\setminus\{\phi\}}, \{(C_v, D_v)\}_{v\in V\setminus\{\phi\}} \right ),
 \end{equation}
 where $(T_0, \SJ_0)$ is an  $\N$-pair and we call it the \emph{initial pair}.
 We associate an $\N^*$-pair to a weighted tree, and  we call it the $\N^*$-pair generated by the weighted tree.
  We show that

 \begin{thm}\label{thm:main-2}  An $\N^*$-pair $(T, \SJ)$ is a  primitive $\N^n$-pair
 if and only if $(T,\SJ)$ is  generated by a  weighted tree with $n$ tops.
 \end{thm}

 \begin{remark}{\rm We show that $(T,\SJ)$  is an $\N^*$-tiling
if and only if in the associated weighted tree, the initial pair is an $\N$-tiling with $\#T_0<\infty$ and
the map $\bdelta$ is constantly zero, or the other round (see Remark \ref{rem:tiling}).
It is folklore that an $\N^n$-tile is also a $\Z^n$-tile, so our construction may shed some light to
the study of $\Z^n$-tiles with $n\geq 2$, an area is almost untouched.
}
\end{remark}

In Theorem \ref{thm:product} we give an explicit formula  for $\N^*$-pairs generated by   weighted trees. (In two dimensional case,   the formula is given by \eqref{eq:delta=0}.)
 Actually, for each $v\in V$, we define two sets $P_v, Q_v\subset \N^n$, and we show that
 $T=\oplus_{v\in V} P_v$ and $\SJ=\oplus_{v\in V} Q_v$, which give further factorizations of $T$ and $\SJ$.

 \begin{example}\label{exam:Niven}{\rm  Niven's characterization of  the $\N^2$-pairs  is
  the case $n=2$ of Theorem \ref{thm:main}.
In the following we describe  an extension process to generate primitive $\N^2$-pairs.

Let $(T_0, \SJ_0)$ be a non-trivial $\N$-pair,  $\delta\in \{0,1\}$,  $\ba=(a_1, a_2)\in ({\mathbb N}\setminus\{0\})^2$,
and let $\Phi_x=(C_1,D_1)$  and $\Phi_y=(C_2,D_2)$ be two interval pairs with size $N_1$ and $N_2$, respectively.
Moreover, $\delta=1$ if $T_0=\{0\}$ and $\delta=0$ if $\SJ_0=\{0\}$.

Regarding $\phi$ as a variable, we have
$T_0(\phi)\SJ_0(\phi)=\N(\phi).$
Applying a second type extension by setting $\phi=x^{a_1}y^{a_2}$, we obtain an $\N^2$-pair
$$
T_1(x,y)=L^\delta_\ba(x,y) T_0(x^{a_1}y^{a_2}), \quad \SJ_1(x,y)=L^{1-\delta}_\ba(x,y)\SJ_0(x^{a_1}y^{a_2}).
$$
Applying two first type extensions to $(T_1,\SJ_1)$, we obtain
\begin{equation}\label{eq:delta=0}
\begin{array}{rl}
T(x,y)&=C_1(x)C_2(y)L_\ba^\delta(x^{N_1},y^{N_2})T_0(x^{a_1N_1}y^{a_2N_2}),\\
 \SJ(x,y)&=D_1(x)D_2(y)L_\ba^{1-\delta}(x^{N_1},y^{N_2})\SJ_0(x^{a_1N_1}y^{a_2N_2}).
 \end{array}
\end{equation}
As a corollary of Theorem \ref{thm:main-2},  a pair $(T,\SJ)$ is a primitive $\N^2$-pair if and only if it is given by \eqref{eq:delta=0}.

}\end{example}

%
%
%
   The paper is organized as follows.
    In Section \ref{sec:unique}, we prove several simple lemmas.
  Section \ref{sec:extension} is devoted to extensions of $\N^*$-pairs.
   Section \ref{sec:pure type}--Section \ref{sec:proof}  investigate the reductions of primitive $\N^*$-pairs;
Theorem \ref{thm:main} is proved in   Section \ref{sec:proof}.
In section \ref{sec:tree} , we introduce weighted trees. Theorem \ref{thm:main-2} is proved   in Section 10.

\section{\textbf{Uniqueness and primitivity}}\label{sec:unique}

  We denote $\bzero_m=(0,\dots, 0)\in \Z^m$; if the dimension is implicit, then we just write $\bzero$.  Let $\prec_g$ be the \emph{lexicographical  order} on $\Z^n$.

\begin{lemma}\label{lem:unique} (i) (Uniqueness) If $(T, \SJ)$  and $(T, \SJ')$ are two  $\N^n$-pairs, then  $\SJ=\SJ'$.

 (ii)  If $T=I_1\times I_2$ and $\SJ=J_1\times J_2$ where
$(I_1,J_1)$ is an $\N^n$-pair and $(I_2, J_2)$ is an $\N^m$-pair, then
$(T,\SJ)$ is an $\N^{n+m}$-pair.
\end{lemma}

 \begin{proof} (i) Let us list $\SJ$ as $(t_k)_{k\geq 1}$ by the following rule:
 for $s,t\in \SJ$, we list $s$ before $t$ if $|s|<|t|$; if $|s|=|t|$, we list $s$ before $t$
 if $s\prec_g t$. We list $\SJ'$ as $(t_k')_{k\geq 1}$ by the same rule.

 Clearly $t_1=t_1'=\bzero$.
  Since $T\oplus (\SJ\setminus\{t_1\})=T\oplus (\SJ'\setminus\{t_1\})$,  the minimum elements (according to the above  rule)
  of $\SJ\setminus\{t_1\}$ and $\SJ'\setminus\{t_1'\}$  coincide, so $t_2=t_2'$.
  Continue this procedure, we obtain  $\SJ=\SJ'$.

 (ii) Denote $x=(x_1,\dots, x_n)$, $y=(y_1,\dots, y_m)$.
  We have
  $
  (I_1\times I_2)(x,y)\cdot (J_1\times J_2)(x,y)=I_1(x)J_1(x)I_2(y)J_2(y)=
 \N^n(x)\N^m(y)=\N^{n+m}(x,y),
  $
  which proves  the second assertion.
  \end{proof}

Actually $(T,\SJ)$ is non-primitive if one of $T$ and $\SJ$ is a Cartesian product.

 \begin{lemma}\label{lem:2.1} If $(T, \SJ)$ is an $\N^{n+m}$-pair such that $T=I_1\times I_2\subset \N^n\times \N^m$, then
 $\SJ$ can be written as $J_1\times J_2\subset \N^n\times \N^m$.
 \end{lemma}
 \begin{proof}  Denote $x=(x_1,\dots, x_n)$, $y=(y_1,\dots, y_m)$.
 The assumptions imply that
  $$
  I_1(x)I_2(y)\SJ(x,y)=\N^{n+m}(x,y).
  $$
  Setting $y=\bzero_m$, we obtain
  $
  I_1(x) \SJ(x,\bzero_m)=\N^n(x).
  $
  Let $J_1(x)=\SJ(x, \bzero_m)$,  then $(I_1, J_1)$ is an $\N^n$-pair. Similarly,
  let $J_2(y)=\SJ(\bzero_n,y)$, then $(I_2, J_2)$ is an $\N^m$-pair.
  By Lemma \ref{lem:unique}(ii), $(I_1\times I_2, J_1\times J_2)$ is an $\N^{n+m}$-pair. Therefore
  $\SJ=J_1\times J_2$ by Lemma \ref{lem:unique}(i).
  \end{proof}

We close this section with several results about $\N$-pairs which will be needed later.

\begin{lemma}\label{lem:sands} (\cite{deB56, Sands})
If $(T, \SJ)$ is an $\N$-pair, then there exists an integer $p\geq 2$, and
an $\N$-pair $(A, B)$ such that
$T=pA$, $\SJ=\{0,1,\dots, p-1\}\oplus pB$
or the other round.
\end{lemma}

\begin{lemma}\label{rem:N-pair} (\cite{deB56, Sands, Tij})
If  $(C, D)$ is a $\{0,1,\dots, N-1\}$-pair with $N\geq 2$,
 then there exists $p\geq 2$ and a $\{0,1,\dots, \frac{N}{p}-1\}$-pair $(\widetilde C,\widetilde D)$
such that
$C=p\widetilde C, D=p\widetilde D+\{0,1,\dots, p-1\}$ or the other round.
\end{lemma}

We denote $C_k\uparrow T$ if  $C_k\subset C_{k+1}$ and $T=\bigcup_{k\geq 1} C_k$. The following lemma is a easy consequence of Proposition \ref{thm:deB}. We leave the simple proof to the reader.
\begin{lemma}\label{cor:CkDk}
Let $(T, \SJ)$ be an $\N$-pair.
 If  $\#T=\#\SJ=\infty$, then there exists a sequence of interval pairs $\{(C_k, D_k)\}_{k\geq 1}$ such that
 $C_k\uparrow T$ and $D_k\uparrow \SJ$, and
 $$
 T=C_k\oplus N_kA_k, \quad \SJ=D_k\oplus N_kB_k
 $$
 for some $A_k, B_k\subset \N$, where $N_k$ is the size of $(C_k, D_k)$.
If $\#T<\infty$, then there exist an integer $N\geq 1$ and $D\subset \N$ such that $T\oplus D=\{0,1,\dots, N-1\}$ and $\SJ=D\oplus N\N.$
\end{lemma}

\section{\textbf{Extensions of $\N^n$-pairs}}\label{sec:extension}

 In this section, we  prove that a finite extension of an $\N$-pair is a primitive $\N^*$-pair.

 Let $\ba\in \N^m$ with $\ba>0$,  recall that
 $L_\ba=\N^m\setminus (\N^m+\ba).$
 Take $z\in \N^m$, and let $k$ be the largest integer such that $z-k\ba\in \N^m$,
 then  $z-k\ba\in L_\ba$. Hence
 \begin{equation}\label{eq:Lba}
 L_\ba\oplus  \N\ba=\N^m.
 \end{equation}
 Especially, $L_p=\{0,1,\dots, p-1\}.$  Denote $x=(x_1,\dots, x_n)$ and $\bar x=(x_2,\dots, x_n)$.

   \begin{lemma}\label{lem:first-type}
The pair  $(T^*,\SJ^*)$  defined in \eqref{eq:ext-1} is an $\N^n$-pair.
   \end{lemma}

 \begin{proof} Without loss of generality, we assume $j=1$ in \eqref{eq:ext-1}. We have
 $$
   \begin{array}{rl}
   T^*(x) \SJ^*(x) &=C(x_1)D(x_1) T(x_1^N, \bar x)\SJ(x_1^{N}, \bar x)\\
   &=(1+x_1+\cdots +x_1^{N-1})\N(x_1^N)\N^{n-1}(\bar x)=\N^{n}(x).
   \end{array}
 $$
   The lemma is proved.
 \end{proof}


\begin{lemma}\label{lem:second-type}
   The pair  $(T^*,\SJ^*)$  defined in \eqref{eq:ext-2} is an $\N^{n+m-1}$-pair.
   \end{lemma}

   \begin{proof} Without loss of generality,  we assume $j=1$ in \eqref{eq:ext-2}.  We have
   $$
   \begin{array}{rl}
   T^*(y,\bar x) \SJ^*(y,\bar x) &=L_\ba(y) T(y^{\ba},\bar x)\SJ(y^{\ba}, \bar x)
   =L_\ba(y)\sum_{k\geq 0} y^{k\ba}\N^{n-1}(\bar x)\\
   &=\N^m(y)\N^{n-1}(\bar x)
   =\N^{n+m-1}(y, \bar x).
   \end{array}
   $$
 (The third equality is due to   \eqref{eq:Lba}.)  The lemma is proved.
   \end{proof}

In the rest of this section,  we  investigate the primitivity property.
We say a binary power series  $T(x_1,\dots, x_n)$ with $n\geq 2$ is  \emph{variable separable},  if
  there exist an integer $m<n$, a permutation $\tau$ on $\{2,\dots, n\}$,
and two sets  $F\in \N^m$, $G\in \N^{n-m}$  such that
\begin{equation}\label{eq:variable}
T(x_1,\dots, x_n)=F(x_1,x_{\tau(2)},\dots, x_{\tau(m)})G(x_{\tau(m+1)},\dots, x_{\tau(n)}).
\end{equation}

The following lemma is a direct consequence of  Lemma \ref{lem:2.1}.

\begin{lemma} An  $\N^n$-pair $(T,\SJ)$ with $n\geq 2$  is primitive  if and only if $T(x)$
   is not  variable separable.
   \end{lemma}

By convention, for $T\subset \N$,
we say $T(x)$ is \emph{variable separable} if and only if $\#T=1$.

 \begin{lemma}\label{lem:separable} Let $T\subset \N^n$. The following three statements are equivalent to each other:

(i) $T(x_1,x_2,\dots, x_n)$ is  variable separable.

(ii) $T(yz,x_2,\dots, x_n)$ is  variable separable.

(iii) $T(x_1^m, x_2,\dots, x_n)$ is  variable separable for some $m\geq 2$.
\end{lemma}

\begin{proof} (i)  implies (ii) and (i) implies (iii) are trivial. In the following,  we prove the other direction implications.

 (ii)$\Rightarrow$(i).  Let us first consider that case $n=1$. If $T(yz)$ is variable separable,
 then $T(yz)=F(y)G(z)$. Set $z=1$, we obtain $T(y)=F(y)G(1)$. Since both $T(y)$ and $F(y)$
 are binary, we obtain that $G(1)=1$.
  Similarly, we have $F(1)=1$. Hence $T(1)=1$, which means $\#T=1$ and $T(x)$ is variable separable.

 Now we assume $n\geq 2$.  Suppose $T(yz, x_2,\dots, x_n)$ is variable separable,  then there is a permutation $\tau$ on $\{2,\dots, n\}$, and two binary power serieses $F$ and $G$ such that either
$$
 T(yz, x_2,\dots, x_n)=F(y, x_{\tau(2)},\dots, x_{\tau(k)})G(z, x_{\tau(k+1)},\dots, x_{\tau(n)}) \text{ with $k\leq n$, or }
$$
$$
T(yz, x_2,\dots, x_n)=F(y, z, x_{\tau(2)},\dots, x_{\tau(k)})G(x_{\tau(k+1)},\dots, x_{\tau(n)}) \text{ with $k<n$. }
$$
Set $y=1$, $z=x_1$,
we obtain
a variable separable factorization of $T(x_1,\dots, x_n)$.

(iii)$\Rightarrow$(i).  If $n=1$, the implication is trivial. Now
we assume that $n\geq 2$. Suppose   $T(x_1^m,x_2,\dots, x_n)$ is variable separable,
 then we have
\begin{equation}\label{eq:T2}
T(x_1^m, x_2,\dots, x_n)=F(x_1,x_{\tau(2)} \dots, x_{\tau(k)})G(x_{\tau(k+1)},\dots, x_{\tau(n)})
\end{equation}
where $k<n$.
Since the power of $x_1$ of any term in $F$ must be a multiple of $m$, we infer that there exists
$F^*\in \N^n$ such that
\begin{equation}\label{eq:T3}
F(x_1,x_{\tau(2)},\dots, x_{\tau(k)})=F^*(x_1^m, x_{\tau(2)},\dots, x_{\tau(k)}).
\end{equation}
 Substituting  \eqref{eq:T3} into \eqref{eq:T2}, and then replacing $x_1^m$ by $x_1$,
 we obtain a variable separable factorization of $T(x_1,\dots, x_n)$.
\end{proof}

\begin{thm}\label{thm:separable}
Let $(T^*,\SJ^*)$ be an extension of  $(T,\SJ)$. If the extension is not illegal, then
 $(T^*,\SJ^*)$ is primitive if and only if $(T,\SJ)$  is primitive.
\end{thm}
\begin{proof} 
(i) Let $(T^*,\SJ^*)$ be a first type extension of $(T,\SJ)$ satisfying
 $T^*(x)=C(x_1)T(x_1^N,{\bar x})$.
If $n=1$, then both $(T,\SJ)$ and $(T^*,\SJ^*)$ are primitive. Now we  assume $n\geq 2$.
By Lemma \ref{lem:separable}, we have
$$
\begin{array}{rl}
  \text{$(T^*, \SJ^*)$ is non-primitive} &\Leftrightarrow  \text{$T^*(x)$ is variable separable}\\
  & \Leftrightarrow  \text{$T(x_1^N,\bar x)=T^*(x)/C(x_1)$ is variable separable}\\
  & \Leftrightarrow \text{ $T(x)$ is variable separable $\Leftrightarrow$ $(T,\SJ)$ is non-primitive.}
\end{array}
$$
The theorem holds in this case.

(ii) Let $(T^*,\SJ^*)$ be a second type extension of $(T,\SJ)$ with parameters $\delta$ and $\ba$.
Let us  assume $\delta=0$, then
$T^*(y,\bar x)=T(y^{\ba}, \bar x),$
where $y=(y_1,\dots, y_m)$. Similar as above,
$(T^*, \SJ^*)$ is non-primitive is equivalent to  $T(x)$ is variable separable.
If $n>1$, it is equivalent to   $(T, \SJ)$ is non-primitive; if $n=1$,
it is equivalent to  $(T, \SJ)=(\{0\},\N)$,  which means the extension is illegal.
The theorem is proved.
\end{proof}

\section{\textbf{$\N^*$-pairs of pure type}}\label{sec:pure type}
Let $\be_1,\dots, \be_n$ be the canonical basis of $\Z^n$.  Let $n\geq 2$, set
 \begin{equation}\label{eq:L0}
L_0 =\{(a_1,\dots, a_n)\in \N^n;\  \text{at least one  $a_j$ is $0$}\}
 \end{equation}
 to be the union of $(n-1)$-faces of $\N^n$.
 An $\N^n$-pair $(T,\SJ)$ with $n\geq 2$ is said to be  of \emph{pure type}, if
either $L_{0}\subset T$ or $L_{0}\subset \SJ$.
In this section, we characterize   $\N^*$-pairs of pure type.

Let $\ba$, $\bb\in \N^n$.
We say $\ba\leq \bb$ if $\bb-\ba\in \N^n$, and $\ba<\bb$ if $\bb-\ba\in (\N\setminus\{0\})^n$.

\begin{lemma}\label{lem:L0} If $(T,\SJ)$ is an $\N^n$-pair  such that $L_0\subset \SJ$,  then
$<$ induces a linear  order on $T$.  We call $\ba$ the \emph{germ} of $T$,
 where $\ba$ is the  minimum of $T\setminus\{\bzero\}$  with respect to $<$.
\end{lemma}

\begin{proof}Denote $a\vee b=\max\{a,b\}$, and $(a_1,\dots, a_n)\vee(b_1,\dots,b_n)=(a_1\vee b_1,\dots, a_n\vee b_n)$.

 Let  $\ba,\bb\in T$. Since $L_0\subset \SJ,$  $\ba+L_0$ and $\bb+L_0$
do not overlap.  If neither  $\ba-\bb>0$  nor $\bb-\ba>0$, then
$\bc:=(\ba-\bb)\vee \bzero\in L_0$ and  $\bc':=(\bb-\ba)\vee \bzero \in L_0$.
It follows that $\bb+\bc=\ba+\bc'$, a contradiction.
\end{proof}


\begin{lemma} \label{lemma:wall}
Let $(T,\SJ)$ be an $\N^n$-pair with $L_0\subset \SJ$, and let $\ba$ be the germ of $T$. Then

 (i) $L_\ba\subset \SJ$;

 (ii) for every integer $k\geq0$,  $(k\ba+L_{\ba})\cap T$ contains at most one element.
\end{lemma}

\begin{proof} (i) is obvious. Now we prove (ii). Suppose on the contrary $u,u'\in (k\ba+L_{\ba})\cap T $. We assume that
$u'<u$ (Lemma \ref{lem:L0}). Then $u-u'\in L_\ba$ since $\bzero \leq u-u'\leq u-k\ba$.
 So $u=(u)+\bzero$ and $u=(u')+(u-u')$ are two different $(T,\SJ)$-decompositions of $u$, which is a contradiction.
\end{proof}

The following  result  is an extension  of Lemma \ref{lem:sands} (the case $n=1$) and
  Niven \cite[Lemma 5]{Niven} ( the case $n=2$) to higher dimensions.

 \begin{thm}\label{thm:pure type} Let $n\geq 2$, let $(T, \SJ)$ be an $\N^n$-pair of pure type with $L_0\subset \SJ$, and  let $\ba$ be the germ of $T$.
  Then there exists an $\N$-pair $(A, B)$ such that
  $$
  T= A\ba, \quad \SJ=B\ba \oplus L_\ba.
  $$
  \end{thm}

   \begin{proof}
  For $k\geq 0$,  denote $\Omega_k=k\ba+L_\ba$, and we call it the $k$-th section of $\N^n$.
   For every integer $k\geq 0$,  we claim that

 $(i)$  $T\cap \Omega_k=\{k\ba\} \text{ or } \emptyset;$  \ $(ii)$ $\SJ\cap \Omega_k=\Omega_k \text{ or } \emptyset.$

By Lemma \ref{lemma:wall},  $\bzero,\ba\in T$ and $L_\ba\subset \SJ$, so the claim holds for $k=0$ .
  Suppose   the claim is false and  let
  $k\geq 1$  be the smallest    integer such that (i) or (ii) fails.

    First,  by the minimality of $k$, we have
  $$ T\cap \bigcup_{j=0}^{k-1}\Omega_j=U\ba,\quad \SJ\cap \bigcup_{j=0}^{k-1}\Omega_j=V\ba\oplus L_\ba$$
  where $U, V\subset\{0,1,\dots, k-1\}$.
Secondly,   we assert that
$$k\ba \not \in T \text{ and } k\not\in U+V.$$
 Suppose $k\ba\in T$, then $(k\ba)+(\bb)$, $\bb\in L_\ba$,  are the $(T,\SJ)$-decomposition
of elements in $k\ba+L_\ba$, and this forces that (i) and (ii) hold for $k$, a contradiction.
By the same reason, we have that  $k\not\in U+V$.

  Let $C=T\cap \Omega_k$ and $D=\SJ\cap \Omega_k$. Then $C$ contains at most one element.
  The set  $\Omega_k$ is covered by
  $$(T\cap \bigcup_{j=0}^{k}\Omega_j))\oplus (\SJ
  \cap \bigcup_{j=0}^{k}\Omega_j)=(U\ba\cup C)\oplus((V\ba+L_\ba)\cup D).$$
 Eliminating the elements apparently outside of $\Omega_k$, we have  that $\Omega_k$ is covered by
 $D\cup (C\oplus L_\ba)\cup((U+V)\ba+L_\ba)$. Since $k\not\in U+V$, we finally obtain
  \begin{equation}\label{eq:CDrom}
  \Omega_k\subset (C\oplus L_\ba)\cup D,
  \end{equation}
  which implies that both $C$ and $D$ are not empty.
%
%

 Let us denote $C=\{\bc\}$. Since $\bc\neq k\ba$, there exists $i$ such that $\bc-\be_i\in \Omega_k$,
 so $\bc-\be_i\in D$ by \eqref{eq:CDrom}. Since $\ba\in T$ and $\ba-\be_i\in \SJ$,
$$
\ba+(\bc-\be_i)=\bc+(\ba-\be_i)
$$
provides two $(T,\SJ)$-decompositions of a same point. This contradiction proves our  claim.

Our claim implies that $T=A \ba $  and $\SJ= B\ba+L_\ba$ for some $A, B\subset \N$.
From  $A \ba \oplus  B \ba \oplus L_\ba=T\oplus \SJ=\N \ba\oplus L_\ba,$
 we deduce that $(A,B)$ is an $\N$-pair. The theorem is proved.
  \end{proof}

\section{\textbf{A key Lemma}}\label{sec:key-lemma}

In this section we prove a crucial lemma which  will be used in Section \ref{sec:marginal}.

Let $n\geq 1$   and write  $z=(z_1,\dots, z_n)$.

\begin{lemma}\label{lem:high-dim}   Let $\ba>0$ be a point in $\N^n$.
Let $A_0, B_0\subset \N$ with $0,1\in A_0$ and $0\in B_0$,
and denote
 $F_0= A_0 \ba$,  $G_0=B_0 \ba \oplus L_\ba.$
Let $\Omega\subset \N \ba$.
 If  $F_1,G_1$ are two subsets of $\N^n$ such that
 \begin{equation}\label{eq:sec}
 F_0(z)G_1(z)+F_1(z)G_0(z)=L_\ba(z)\Omega(z),
 \end{equation}
 then
 $F_1=A_1 \ba$ and $G_1=B_1 \ba\oplus L_\ba$
 for some $A_1, B_1\subset \N$.
\end{lemma}


\begin{proof} The assumption \eqref{eq:sec} implies that
$$
(F_0\oplus G_1)\cup (F_1\oplus G_0)=\Omega\oplus L_\ba \text{ and } (F_0\oplus G_1)\cap(F_1\oplus G_0)=\emptyset.
$$

Denote $\Omega=( d_k\ba)_{k\geq 1},$
 where $(d_k)_{k\geq 1}$ is  an increasing sequence in $\N$.
We  denote
$\Omega_k:=d_k\ba+L_\ba$ and call it
the $k$-th \emph{section} of $\Omega+L_\ba$.
We claim that:

\smallskip

\textit{Claim 1.   For all $k\geq 0$,  $G_1\cap \Omega_k=\emptyset$ or   $\Omega_k$.}

\smallskip

We prove the claim by induction. Clearly the claim holds for $k=1$.
Let $k\geq 2$ and assume  that    $G_1\cap \Omega_j=\emptyset$ or   $\Omega_j$ holds for $j<k$. Then
 \begin{equation}\label{eq:G1}
G_1\cap (\Omega_1\cup \cdots \cup \Omega_{k-1})=V\ba+L_\ba,
\end{equation}
 where $V\subset\{d_1, d_2,\dots, d_{k-1}\}$. Hence,  for $j<k$,
\begin{equation}\label{eq:G-omega}
(F_0\oplus G_1)\cap \Omega_j=(F_0\oplus (\{g_0,\dots, g_h\}\ba+L_\ba))\cap \Omega_j=\emptyset \text{ or }\Omega_j.
\end{equation}
Consequently, for $j<k$, as the complement of the right hand side of \eqref{eq:G-omega},
\begin{equation}\label{eq:F-omega}
(F_1\oplus G_0)\cap \Omega_j=\emptyset \text{ or }\Omega_j.
\end{equation}
Let $f^*$ be the smallest element in $F_1$ such that $f^*\not\in \N\ba$. By \eqref{eq:G-omega} and \eqref{eq:F-omega},
\begin{equation}\label{eq:f-star}
f^*\not\in \Omega_1\cup\cdots\cup \Omega_{k-1} \ (\text{if $f^*$ exists}).
\end{equation}

Suppose that Claim 1 is false for $k$.   Denote $C:=G_1\cap \Omega_k$.
We assert that
\begin{equation}\label{eq:C}
(F_0\oplus G_1)\cap \Omega_k=C.
\end{equation}
Notice that
$$(F_0\oplus G_1)\cap \Omega_k=[(F_0\oplus (\{g_0,\dots, g_h\}\ba+L_\ba))\cap \Omega_k]
\cup [(F_0\oplus C)\cap \Omega_k].$$
The second term on the right hand side is $C$,  so the first term must be the empty set, which implies \eqref{eq:C}.

By \eqref{eq:C} and \eqref{eq:F-omega}, we have that  $(F_1\oplus G_0)\cap \Omega_k\neq \emptyset$, so there exists
$f\in F_1$ such that $(f+  G_0)\cap \Omega_k\neq \emptyset$. The intersection is not $\Omega_k$ implies that
$f\not\in \N\ba$, so we have  $f\in \Omega_k$  since $f\geq f^*$.  Since $F_1\cap \Omega_k$ contains at most
one element, we conclude that $f=f^*$.

 Let $i$ be the index in $\{1,\dots, n\}$   such that
$f-\be_i\in \Omega_k$. Then $f-\be_i$ is not covered by $f+G_0$, and hence
it is not covered by $F_1+G_0$. It follows that $f-\be_i\in C$. Therefore, on one hand,  we have
$$
\ba+(f-\be_i)\in F_0+C\subset F_0+G_1;
$$
on the other hand,
$
f+(\ba-\be_i)\in F_1+L_\ba\subset F_1+G_0.
$
This is a contradiction, and Claim 1 is proved. Therefore, \eqref{eq:G1} and \eqref{eq:f-star} holds for all $k$, which imply that
$G_1=B_1\ba\oplus L_\ba$, $F_1=A_1\ba$ for some $A_1, B_1\subset \N$. The lemma is proved.
\end{proof}

\section{\textbf{Marginal pairs of $\N^*$-pairs}}\label{sec:marginal}
Let $j_1,\dots, j_m(1\leq m<n)$ be a subsequence of $1,\dots, n$.
We call
$Q=\N\be_{j_1}\oplus \cdots \oplus \N\be_{j_m}$
  an $m$-\emph{face} of $\N^n$. Define $\rho:Q\to \N^m$ as
  $
  \rho\left (\sum_{\ell=1}^m c_\ell \be_{j_\ell}\right )=(c_1,\dots, c_m).
  $
  We denote
  \begin{equation}
  T_Q=T\cap Q, \quad \SJ_Q=\SJ\cap Q,
  \end{equation}
  and  call $(\rho(T_Q), \rho(\SJ_Q))$ the \emph{marginal pair induced by $Q$}.

\begin{lemma}\label{lem:marginal}
Let $(T, \SJ)$ be an $\N^n$-pair and $Q$ be an $m$-face of $\N^n$. Then $(T_Q, \SJ_Q)$ is a $Q$-pair.
Consequently, the marginal pair $(\rho(T_Q), \rho(\SJ_Q))$ is an $\N^m$-pair.
\end{lemma}

\begin{proof}  In $T(x)\SJ(x)=\N^n(x)$,  setting  $x_k=0$ if $k\not\in \{j_1,\dots, j_m\}$,  we obtain
  $T_Q(x)\SJ_Q(x)=Q(x)$, which implies that  $T_Q\oplus \SJ_Q=Q$. The second assertion is obvious.
\end{proof}

If $(T,\SJ)$ is an $\N$-pair with  $\{0,1,\dots, p-1\}\subset \SJ$ and $p\in T$, we call $p$ the \emph{germ}
of $T$.

\begin{thm}\label{thm:marginal}
Let $(T, \SJ)$ be an $\N^n$-pair and  $Q=\N\be_1\oplus\cdots\oplus\N\be_m$ with $1\leq m\leq n$.
If the marginal pair $(\rho(T_Q), \rho(\SJ_Q)$ is of pure type in case of $m>1$, then there exists an $\N^{n-m+1}$-pair $(\widetilde T, \widetilde \SJ)$ such that
  \begin{equation}\label{eq:extension-1}
  \begin{array}{ll}
  T(x)= \widetilde T(x_1^{a_1}\cdots x_m^{a_m}, x_{m+1}, \dots, x_n),\\
 \SJ(x)=L_\ba(x_1,\dots, x_m)\widetilde \SJ(x_1^{a_1}\cdots x_m^{a_m}, x_{m+1}, \dots, x_n),
 \end{array}
 \end{equation}
where $\ba=(a_1,\dots, a_m)$ is the germ of $\rho(T_Q)$.
  \end{thm}

We note that  \eqref{eq:extension-1} is a second type extension   if $m>1$, and is of first
type if $m=1$. Moreover, if $m=n$, then Theorem \ref{thm:marginal} becomes Theorem \ref{thm:pure type} or Lemma \ref{lem:sands}.

 \begin{proof} By Theorem \ref{thm:pure type} or Lemma \ref{lem:sands}, there exists an $\N$-pair $(A,B)$ with $1\in A$ such that \begin{equation}\label{eq:rho}
\rho(T_Q)=A\ba,\quad \rho(\SJ_Q)=B \ba  \oplus L_\ba,
\end{equation}
(Remember that in case of $m=1$,  $L_\ba=\{0,1,\dots, \ba-1\}$.)  Let $\bm\in \N^{n-m}$. Define
$$T_\bm=\{u\in \N^m;~(u,\bm)\in T\}, \quad \SJ_\bm=\{v\in \N^m;~(v,\bm)\in \SJ\}.$$
 Denote $\bar {\bar x}=(x_{m+1},\dots, x_n)$, then
\begin{equation}\label{eq:deco-new}
T(x)=\sum_{\bp\in \N^{n-m}} T_\bp(x_1,\dots, x_m)\bar {\bar x}^\bp,
\quad \SJ(x)=\sum_{\bq\in \N^{n-m}}\SJ_\bq(x_1,\dots, x_m) \bar {\bar x}^\bq .
\end{equation}
We claim that for all $\bm\in \N^{n-m}$, it holds that
\begin{equation}\label{needle-1-new}
  \N^m(x_1,\dots,x_m)  =  \sum_{\bp+\bq=\bm} T_\bp(x_1,\dots, x_m) \SJ_{\bq}(x_1,\dots, x_m).
\end{equation}
Clearly
$\N^n(x)=\N^m(x_1,\dots, x_m)\sum_{\bd\in \N^{n-m}} \bar{\bar x}^\bd.$
On the other hand, by \eqref{eq:deco-new},
$$
\N^n(x)=T(x)\SJ(x)
= \sum_{\bd\in \N^{n-m}} \sum_{\bp+\bq=\bd} T_\bp(x_1,\dots, x_m) \SJ_{\bq}(x_1,\dots, x_m)\bar{\bar x}^\bd.
$$
Comparing the terms in the above two equations with the factor $\bar{\bar x}^{\bm}$, we obtain \eqref{needle-1-new}.

Denote $\bar x=(x_1,\dots, x_m)$. We shall prove by induction that
\begin{equation}\label{needle-2-new}
T_\bm(\bar x)=f_\bm(\bar x^\ba), \quad \SJ_\bm(\bar x)=L_\ba(\bar x)G_\bm(\bar x^\ba),
\end{equation}
for some $f_\bm, G_\bm\subset \N$.

For $\bm=\bzero$,  $(T_\bzero, \SJ_\bzero)=(\rho(T_Q), \rho(\SJ_Q))$.
In this case \eqref{needle-2-new} holds by  \eqref{eq:rho}.

Suppose \eqref{needle-2-new}  holds if we replace $\bm$ by any point $\bm'\leq  \bm$ and $\bm'\neq \bm$.
We are going to show that \eqref{needle-2-new} holds for $\bm$.

 If $\bp+\bq=\bm$ and $\bp\neq \bzero$, $\bq\neq \bzero$, then \eqref{needle-2-new} holds for  $\bp$ and $\bq$, so
 \begin{equation}\label{eq:Hpq}
 T_\bp(\bar x)\SJ_{\bq}(\bar x)=L_\ba(\bar x)f_{\bp}(\bar x^\ba)G_{\bq}(\bar x^{\ba}).
 \end{equation}
 By  \eqref{needle-1-new}, we have
 \begin{equation}\label{eq:pq}
  \sum_{\bp+\bq=\bm} T_\bp(\bar x) \SJ_{\bq}(\bar x)=L_\ba(\bar x)\N(\bar x^{\ba}),
 \end{equation}
 Subtracting  \eqref{eq:Hpq} from \eqref{eq:pq}, we obtain that   there exists   $\Omega\subset \N$ such that
 $$
 T_{\bzero}(\bar x)\SJ_\bm(\bar x)+T_\bm(\bar x)\SJ_\bzero(\bar x)=L_\ba(\bar x)\Omega(\bar x^\ba).
 $$
 Hence, by Lemma \ref{lem:high-dim}, we obtain \eqref{needle-2-new}.

 Finally, substituting \eqref{needle-2-new} into \eqref{eq:deco-new}, we obtain that  (\ref{eq:extension-1})
 holds for some $\widetilde T, \widetilde \SJ\in \N^{n-m+1}$.
 It is an easy matter to show that $(\widetilde T, \widetilde \SJ)$ is an $\N^{n-m+1}$-pair. The  theorem is proved.
  \end{proof}

\section{\textbf{$1$-face reduction of  $\N^*$-pair}}\label{sec:1-face}

Let $(T, \SJ)$ be a primitive $\N^n$-pair. We say $(T,\SJ)$ is of \emph{class} ${\cal F}_0$,
if
$$
 T\cap \N\be_j=\{\bzero\} \text{ or }\N\be_j, \quad  \text{for all }j\in \{1,\dots, n\}.
$$
In this section we show that any primitive $\N^n$-pair is a finite extension
of a pair of class ${\cal F}_0$.

Denote  $\bar{x}=(x_2,\dots, x_n)$.
Set $Q=\N\be_1$ and $\ba=p$ in Theorem \ref{thm:marginal}, we obtain

\begin{lemma}\label{thm:sands} Let $(T,\SJ)$ be an $\N^n$-pair.  If $p\be_1\in T$ and $\{0,1,\dots, p-1\}\be_1\subset \SJ$, then
there exists an $\N^n$-pair  $(\widetilde T, \widetilde \SJ)$ such that
$$
T(x)=\widetilde T(x_1^p, \bar{x}), \quad
\SJ(x)=(1+x_1+\cdots+x_1^{p-1})\widetilde \SJ(x_1^p, \bar{x}).
$$
\end{lemma}

Using the above lemma repeatedly, we   get   `bigger' factors of $T(x)$ and $\SJ(x)$.

\begin{pro}\label{lem:CD} Let $(T,\SJ)$ be an $\N^n$-pair.
Denote $T\cap \N\be_1=A\be_1,$ $\SJ\cap \N\be_1=B\be_1$. Let $(C,D)$ be a $\{0,1,\dots, N-1\}$-pair such that
$$A=C\oplus NA_1, \quad  B=D\oplus NB_1$$
for an $\N$-pair $(A_1, B_1)$. Then there exists an $\N^n$-pair $(\widetilde T, \widetilde \SJ)$ such that
\begin{equation}\label{eq:1-dim-reduce}
T(x)=C(x_1)\widetilde T(x_1^N, \bar x), \quad \SJ(x)=D(x_1)\widetilde \SJ(x_1^N, \bar x).
\end{equation}
\end{pro}

\begin{proof}
We prove the result by induction on $N$. If $N=1$, the proposition holds trivially.
Assume $N\geq 2$ and that the proposition holds for any pair and any positive integer less than $N$.

Assume $1\in D$ without loss of generality.

 If $C=\{0\}$, let $h$ be the smallest   integer such that $h\not\in B$.
 Then by Lemma \ref{thm:sands},
 $$T(x)=T^*(x_1^h, \bar x), \quad \SJ(x)=(1+x_1+\dots+x_1^{h-1})\SJ^*(x_1^h,\bar x).$$
 Since   $C(x_1)=1$, $D(x_1)=1+x_1+\cdots+x_1^{N-1}$ and $N|h$, set
 $$\widetilde T(x)=T^*(x_1^{h/N}, \bar x),\quad  \widetilde \SJ(x)=\sum_{j=0}^{h/N-1} x_1^{j} \SJ^*(x_1^{h/N},\bar x),$$
 we obtain \eqref{eq:1-dim-reduce}.

 Now suppose $\#C\geq 2$. Let $p$ be the smallest element of $C$ other then $0$. Then $p\geq 2$.
On one hand,  by Lemma \ref{rem:N-pair},  there exists a $\{0,1,\dots, {N}/{p}-1\}$-pair  $(\widetilde C, \widetilde D)$  such that
\begin{equation}\label{reduction_1}
C(x_1)=\widetilde C(x_1^p), \quad D(x_1)=(1+x_1+\dots+x_1^{p-1})\widetilde D(x_1^p).
\end{equation}
On the other hand, by  Lemma \ref{thm:sands},
 there exists an $\N^n$-pair $(T^*, \SJ^*)$ such that
\begin{equation}\label{reduction_2}
T(x)= T^*(x_1^p,\bar x), \quad \SJ(x)=(1+x_1+\cdots+x_1^{p-1}) \SJ^*(x_1^p,\bar x).
\end{equation}


Denote $T^*\cap \N\be_1=A^*\be_1,$ $\SJ^*\cap \N\be_1=B^*\be_1$.
By \eqref{reduction_2}, we have
 $A=pA^*$ and $B=\{0,1,\dots, p-1\}+pB^*$, which together with \eqref{reduction_1} imply that
$$A^*=\widetilde C\oplus (N/p)A_1, \quad B^*=\widetilde D\oplus (N/p)B_1.$$
So by induction hypothesis, there exists an $\N^n$-pair $(\widetilde T, \widetilde \SJ)$ such that
$$
T^*(x)=\widetilde C(x_1) \widetilde T(x_1^{N/p},\bar x), \quad
\SJ^*(x)=\widetilde D(x_1) \widetilde \SJ(x_1^{N/p},\bar x).
$$
 This together with \eqref{reduction_1} and \eqref{reduction_2} imply \eqref{eq:1-dim-reduce}.
The lemma is proved.
\end{proof}

The following theorem
is proved by Niven (\cite[Lemma 4]{Niven}) in the two dimensional case.

\begin{thm}\label{thm:finite}
If $(T, \SJ)$ be a primitive $\N^n$-pair with $n\geq 2$, then either
$T\cap \N\be_1$ or $\SJ\cap \N\be_1$ is   finite.
\end{thm}

\begin{proof} Let $Q=\N\be_2\oplus \cdots \oplus \N\be_n$. Recall that
$T_Q=T\cap Q$, $\SJ_Q=\SJ\cap Q.$
Write
$T\cap \N\be_1=A\be_1$,  $\SJ\cap \N\be_1=B\be_1.$
Then $A\oplus B=\N$ and $T_Q\oplus \SJ_Q=Q$ (Lemma \ref{lem:marginal}).
Define
$\rho(a_1,\dots, a_n)=(a_2,\dots, a_n).$

Let $(C, D)$ be a $\{0,1,\dots, N-1\}$-pair satisfying the assumptions of Proposition \ref{lem:CD}.
Then there exists $H\subset \N^n$ such that
\begin{equation}\label{eq:CH}
T(x)=C(x_1) H(x).
\end{equation}
Setting $x_1=0$, we obtain  $T(0, x_2,\dots, x_n)=H(0,x_2,\dots, x_n)$,  which implies that $H\cap Q=T\cap Q=T_Q.$
So  $T_Q\subset H$, and   by \eqref{eq:CH}, we have
\begin{equation}\label{eq:CTQ}
C\times \rho(T_Q)\subset T.
\end{equation}

Suppose on the contrary that both $A$ and $B$ are infinite sets.
By Lemma \ref{cor:CkDk}, there exists a sequence of interval pairs $\{C_k, D_k)\}_{k=1}^\infty$   such that
$C_k\uparrow A$ and $D_k\uparrow B$ and $(C_k, D_k)$ satisfies the assumptions of Proposition \ref{lem:CD}.
This together with \eqref{eq:CTQ} imply that $A\times \rho(T_Q)\subset T$. Similarly, we have $B\times \rho(\SJ_Q) \subset \SJ$.

Since $(A\times \rho(T_Q), B\times \rho(\SJ_Q))$ is an $\N^n$-pair (Lemma \ref{lem:unique}),
  the above two inclusion relations imply that
$T=A\times \rho(T_Q)$ and $\SJ=B\times \rho(\SJ_Q),$
 which imply $(T,\SJ)$ is not primitive.
This contradiction proves the theorem.
\end{proof}

 \begin{thm}\label{thm:1-face} Let  $(T, \SJ)$ be a primitive $\N^n$-pair.
 Then there exists a primitive $\N^n$-pair $(\widetilde T, \widetilde \SJ)$ in ${\cal F}_0$
 such that $(T, \SJ)$ is a finite (first type) extension of $(\widetilde T, \widetilde \SJ)$.
 \end{thm}

 \begin{proof}
  We first reduce $(T, \SJ)$ along the variable $x_1$.
  Denote $ T\cap \N\be_1=A\be_1, \SJ\cap \N\be_1=B\be_1.$
 Then either $A$ or $B$ is finite (Theorem \ref{thm:finite}). Assume   $A$ is finite.
 By Lemma \ref{cor:CkDk}, there exist   $D\subset B$
 such that $A\oplus D=\{0,1,\dots,N-1\}$ and $B=D+N\N$.
  Hence, by Proposition \ref{lem:CD},
 there exists an  $\N^n$-pair $(T_1, \SJ_1)$ such that
 \begin{equation}\label{eq:CD}
 T(x)=A(x_1)T_1(x_1^N,\bar{x}),\quad   \SJ(x)=D(x_1)\SJ_1(x_1^N,\bar{x}).
 \end{equation}

We claim that: $(i)$ $T_1\cap \N\be_1=\{\bzero\}$;  $(ii)$ $T_1\cap \N\be_j=T\cap \N\be_j$ for $j\neq 1$.

From $T\cap \N\be_1=A\be_1$ we infer that $T(x_1,\bzero)=A(x_1)$,
 which together with \eqref{eq:CD} imply that $T_1(x_1^N, \bzero)=1$. So (i) holds.
Pick $j>1$,
by \eqref{eq:CD},  we have $T(\bzero,x_j,\bzero)=T_1(\bzero,x_j,\bzero),$ which means that $T\cap \N\be_j=T_1\cap \N\be_j$.
This proves (ii).

Next, we reduce $(T_1, \SJ_1)$
 along the variables $x_2,\dots, x_n$ one by one, and finally we obtain   an $\N^n$-pair $(\widetilde T, \widetilde \SJ)$ of class ${\cal F}_0$.
  The theorem is proved.
\end{proof}

\section{\textbf{Proof of Theorem \ref{thm:main}}}\label{sec:proof}

 The following lemma provides the last ingredient for proving  Theorem \ref{thm:main}.

  \begin{lemma}\label{lem:Q_0}
  Let $n\geq 2$ and  $(T, \SJ)$ be a primitive $\N^n$-pair of class ${\cal F}_0$.
  Then there exists an $m$-face $Q_0$ with $2\leq m\leq n$  such that the marginal pair induced by $Q_0$ is of pure type.
  \end{lemma}

\begin{proof} Without loss of generality,  we  assume that $T\cap \N \be_j=\{\bzero\}$ holds for at least one $j$.
  By rearranging the order of variables, we may assume that
  $$
  T\cap \N\be_j=\{\bzero\}  \text{ for } j=1,\dots, M, \text{ and }   T\cap \N\be_j=\N\be_j \text{ otherwise},
  $$
  where $1\leq M\leq n$.  Denote $Q=\oplus_{j=1}^M \N\be_j$ and $Q'=\oplus_{j=M+1}^n \N\be_j$. (If $M=n$, we set $Q'=\{\bzero\}$.)
  Then $(T_Q, \SJ_Q)$ is a $Q$-pair and  $(T_{Q'}, \SJ_{Q'})$ is a $Q'$-pair.

 We claim that either $\# T_Q\geq 2$ or $\# \SJ_{Q'}\geq 2$.
  Suppose on the contrary that $\# T_Q=1$ and  $\#\SJ_{Q'}=1$. Then
  $T_Q=\SJ_{Q'}=\{\bzero\},$ which implies that  $\SJ_Q=Q\subset \SJ$ and $T_{Q'}=Q'\subset T.$
 It follows  that
 $(T, \SJ)=(Q', Q)$ and it is not primitive, which contradicts our assumption.

 Without loss of generality, let us    assume   that  $\# T_Q\geq 2$.
Let  $x^{\ba'}$ be a term  of $T_Q(x)-1$ such that
    the number of non-zero entries  of $\ba'$
 attains the minimum. Write $x^{\ba'}$ as
  $$x^{\ba'}=x_{j_1}^{r_1}\cdots x_{j_m}^{r_m}, \quad \text{ where } j_1,\dots, j_m\in \{1,\dots, M\}.$$
 Then $m\geq 2$ since $T\cap \N\be_j=\{\bzero\}$ for $1\leq j\leq M$.
Let $Q_0=\N\be_{j_1}\oplus\cdots\oplus \N\be_{j_m}$, which is the smallest face containing ${\ba'}$. Let $(f, g)$
   be the marginal pair induced by $Q_0$,
  then $(f, g)$  is an $\N^m$-pair of pure type by the minimality of $m$. The lemma is proved.
\end{proof}

\begin{proof}[\textbf{Proof of Theorem \ref{thm:main}.}] By Theorem \ref{thm:separable},
we only need show that any primitive $\N^*$-pair is a finite extension of an $\N$-pair.
Let $(T,\SJ)$ be a primitive $\N^n$-pair.
  We prove by induction on the dimension of  $(T,\SJ)$.
  If $n=1$, the theorem is trivially true.

  Assume  $n\geq 2$.
  By Theorem \ref{thm:1-face}, there exists a primitive $\N^n$-pair $(T_1, \SJ_1)$ of class ${\cal F}_0$ such that
  $(T, \SJ)$ is a finite extension of it.
  By Lemma \ref{lem:Q_0}, $(T_1, \SJ_1)$ has  an $m$-face $Q_0$ with $m\geq 2$ such that the marginal pair
     induced by $Q_0$ is of pure type.  So
  by Theorem \ref{thm:marginal}, there exists a primitive $\N^{n-m+1}$-pair $(\widetilde T, \widetilde \SJ)$
  such that $(T_1,\SJ_1)$ is a second type extension of $(\widetilde T, \widetilde \SJ)$.
 Finally, by  induction hypothesis, $(\widetilde T, \widetilde \SJ)$ is a finite extension of an $\N$-pair,
 say $(T_0, \SJ_0)$. Then $(T_0, \SJ_0)\to (\widetilde T, \widetilde \SJ) \to (T_1, \SJ_1)\to (T, \SJ)$
indicates the extension process from  $(T_0, \SJ_0)$ to  $(T, \SJ)$.
\end{proof}

\section{\textbf{Weighted tree of $\N^*$-pair}}\label{sec:tree}
In this section, we  study  weighted trees and the associated $\N^*$-pairs.

\subsection{Weighted tree}
A graph  $(V, \Gamma)$ with  node set $V$ and edge set $\Gamma$
  is a \emph{tree} if  it  is connected  and has no  cycle.
A \emph{rooted tree} is a triple $(V, \Gamma, \phi)$ where $(V,\Gamma)$ is a tree and $\phi\in V$;
we call $\phi$ the  \emph{root}
of the tree.

Let $(V, \Gamma)$ be a finite tree with root $\phi$.
The \emph{level} of $v$, denoted by $|v|$,  is   the length of the (unique) path joining $v$ and
$\phi$.
A node $v$ is called an \emph{offspring} of $u$, if there is an edge joining $u$ and $v$, and $|v|=|u|+1$; meanwhile we call $u$ the \emph{parent} of $v$.  For $u,v\in V$, we use $[u,v]$ to denote the edge joining $u$ and its offspring $v$. We call $u$ an \emph{ancestor} of $v$ if
 there is a path joining $u$ and $v$, and   $|u|<|v|$.
A node $v$ is called a \emph{top} if
it has no offspring. We denote by $V_0$ the set of tops.
We shall use the nodes as variables of power series.

\begin{defi}\label{def:tree}{\rm We call the quadruple
\begin{equation}\label{eq:w-tree}
\left ((T_0,\SJ_0), \bdelta , \balpha, \Phi\right )
\end{equation}
a weight of the tree $(V,\Gamma)$,
where $(T_0, \SJ_0)$ is an $\N$-pair, called the \emph{initial pair},
$$\bdelta: V\setminus V_0\to \{0,1\},\  \balpha: V\setminus \{\phi\} \to \N\setminus \{0\}, \ \Phi: V\setminus \{\phi\}\to \{\text{interval pairs}\}$$
are three maps, and in addition, $\bdelta_\phi=1$ if $T_0=\{0\}$ and
$\bdelta_\phi=0$ if $\SJ_0=\{0\}$.
}
\end{defi}

We call $\#V-\#V_0$ the \emph{norm} of the tree.
Two tops are said to be equivalent, if they share the same parents. An equivalent class of $V_0$ is called a \emph{branch} of $V_0$.

\subsection{Constructing $\N^*$-pairs from weighted trees} Now we define the pair generated by a weighted tree inductively on the norm of the tree.

Notice that $\Lambda_0=(\{\phi\}, \emptyset)$ is the only tree with norm $0$. The weight of $\Lambda_0$ only
consists of the initial pair, which we denote by $(T_0, \SJ_0)$.
We define the $\N^*$-pair associate with the weighted tree $\Lambda_0$ to be
$(T_0, \SJ_0)$.

 Let $q\geq 1$.
Suppose for any weighted tree  with norm less than $q$, we have associated an
  $\N^m$-pair  with it, where $m$ is the number of tops of the tree.

Let $(V,\Gamma)$ be a weighted tree with  norm $q$.
Let $V_0=\{x_1,\dots,x_n\}$.
Choose  a branch   $\{x_1,\dots, x_p\}$ of $V_0$ and denote their parents by $z_1$.
Let $(\widetilde V, \widetilde \Gamma)$ be  the subtree of $(V, \Gamma)$ obtained
by deleting the nodes $x_1,\dots, x_p$ and the edges $[z_1,x_1], \dots, [z_1, x_p]$.
Then
\begin{equation}\label{eq:VV0}
\widetilde V=V\setminus \{x_1,\cdots,x_p\}, \quad \widetilde V_0=\{z_1, x_{p+1},\dots, x_n\}.
\end{equation}

We  define the weight of $(\widetilde V, \widetilde \Gamma)$ to be the restriction of the
 weight $((T_0,\SJ_0), \bdelta, \balpha, \Phi)$ on it. (Now $z_1$ is a top of the subtree, so $\bdelta_{z_1}$
 is not needed in the restricted weight.)

 Since $\#\widetilde V-\# \widetilde V_0=q-1$, by the induction hypothesis, there is an $\N^{n-p+1}$-pair
 $(\widetilde T, \widetilde \SJ)$ associated with the weighted tree of $(\widetilde V, \widetilde \Gamma)$.

Denote $\delta=\bdelta_{z_1}$.
 For $j=1,\dots, p$, let $a_j=\balpha_{x_j}$,
 $(C_j,  D_j)=\Phi_{x_j}$,
 and let $N_j$ be the size of $( C_j,  D_j)$.
Denote $\bar x=(x_{p+1},\dots, x_n)$ and $\ba=(a_1,\dots, a_p)$.

 First,  applying  an second type extension with parameters $\delta$ and
     $\ba$  to $(\widetilde T, \widetilde \SJ)$, we obtain
$$F(x)=L^{ \delta}_\ba(x_1,\dots, x_p) \widetilde T(\prod_{j=1}^p x_j^{a_j}, \bar x),
\quad
 G(x)=L^{1- \delta}_\ba(x_1,\dots, x_p)\widetilde \SJ(\prod_{j=1}^p x_j^{a_j}, \bar x).$$
Secondly, applying $p$ extensions of the first type consecutively to the pair $(F,G)$, along the variables $x_1,\dots, x_p$, we obtain
\begin{equation}\label{eq:treee}
\begin{array}{l}
T(x)=\left ( \prod_{j=1}^p  C_j(x_j) \right )L^{ \delta}_\ba(x_1^{N_1},\dots, x_p^{N_p}) \widetilde T(\prod_{j=1}^p x_j^{a_j N_j},\bar x),\\
\SJ(x)=\left (\prod_{j=1}^p D_j(x_j) \right )L^{1- \delta}_\ba(x_1^{N_1},\dots, x_p^{N_p})\widetilde \SJ(\prod_{j=1}^p x_j^{a_j N_j}, \bar x).
\end{array}
\end{equation}
We call $(T,\SJ)$ the \emph{$\N^*$-pair generated by the weighted tree} \eqref{eq:w-tree}.

To show  the generated pair is well-defined, we need to show that $(T,\SJ)$ is independent of
the choice of the branch $\{x_1,\dots, x_p\}$. We will do this in Theorem \ref{thm:product}.


\begin{remark}\label{rem:tiling}{\rm
Clearly, $\#T<\infty$ if and only if  $\#T_0<\infty$ and $\bdelta$ is constantly $0$;
a similar result holds for  $\#\SJ$.
Also, $(T,\SJ)$  is of class ${\cal F}_0$ if and only if $\Phi_v=(\{\bzero\},\{\bzero\})$ for all tops.}
\end{remark}

\subsection{A closed formula of $(T,\SJ)$}
  Let $(T, \SJ)$ be the pair defined by \eqref{eq:treee}.
We define a map $\mu: V \times V_0 \to \N $ as follows.
Pick  $v\in V$ and  $x^*\in V_0$.

 If $v$ is an ancestor of $x^*$, let $v, v_1,\dots, v_{k-1}, v_k=x^*$
be the path from $v$ to  $x^*$;
for $j=1,\dots,k,$  let $c_j=\balpha_{v_j}$
and $M_j$ be the size of  $\Phi_{v_j}$.
 We define
\begin{equation}\label{mea-number}
\mu(v, x^*)=\left \{
\begin{array}{cl}
\prod_{j=1}^k c_jM_j, & \text{ if $v$ is an ancestor of $x^*$;} \\
1, & \text{ if $v=x^*$;}\\
0, & \text{ otherwise}.
\end{array}\right .
\end{equation}
Moreover, we define
$\mu(v)=(\mu(v,x_1),\dots, \mu(v,x_n)).$

For each $v\in V$, we define two power series  $P_v$ and $Q_v$ as follows.  Denote $\Phi_v=(C_v, D_v)$.

If $v$ is a top, we set
 \begin{equation}\label{eq:root}
 P_v(x)=C_v(v),\quad Q_v(x)=D_v(v).
 \end{equation}
  (If  $V=\{\phi\}$, we make the convention
 $P_\phi(\phi)=T_0(\phi)$ and $Q_\phi(\phi)=\SJ_0(\phi)$.)

If $v$ is not a top, let  $z_1,\dots, z_m$ be the offsprings of $v$.
Denote  $\bb=(\balpha_{z_1},\dots,\balpha_{z_m})$ and $M_j$ be the size of
$\Phi_{z_j}$. (If $v=\phi$, we set $C_\phi=T_0$ and $D_\phi=\SJ_0$.) We define
\begin{equation}\label{eq:Pv}
\begin{array}{l}
P_v(x)=C_v(x^{\mu(v)}) L^{\bdelta_v}_{\bb}(x^{M_1\mu(z_1)},\dots, x^{M_m\mu(z_m)}),\\
Q_v(x)=D_v(x^{\mu(v)}) L^{1-\bdelta_v}_{\bb}(x^{M_1\mu(z_1)},\dots, x^{M_m\mu(z_m)}).
\end{array}
\end{equation}

\begin{thm}\label{thm:product}
Let $(T, \SJ)$ be an $\N^n$-pair  generated by the tree $(V,\Gamma)$ with weight (\ref{eq:w-tree}).
Then $P_v(x)$ and $Q_v(x)$ defined by \eqref{eq:root} and \eqref{eq:Pv} are binary power series, and
\begin{equation}\label{eq:Pvv}
T(x)=\prod_{v\in V} P_v(x),  \quad \SJ(x)=\prod_{v\in V} Q_v(x)
\end{equation}
\end{thm}

\begin{proof} Denote $q=\#V-\#V_0$.  We prove the result by induction on $q$.
If $q=0$,  then $V=\{\phi\}$, and  (\ref{eq:Pvv}) holds  by our convention. Now we assume that $q\geq 1$.

Let $(\widetilde V, \widetilde \Gamma)$ be the   subtree of $(V, \Gamma)$ given by \eqref{eq:VV0},
and the weight is the restricted weight.
According to this restricted weight,
for any $v\in \widetilde V$ and $z\in \widetilde V_0=\{z_1, x_{p+1}, \dots, x_n\}$,
 we can  define a map  $\widetilde{\mu}(v,z)$ as    \eqref{mea-number},
 and   power series   $\widetilde P_v, \widetilde Q_v$ as \eqref{eq:root} and \eqref{eq:Pv}.

Let $(\widetilde T, \widetilde \SJ)$ be the pair generated by the weighted tree $(\widetilde V, \widetilde \Gamma)$.
Denote $\bar x=(x_{p+1},\dots, x_n)$.
By induction hypothesis, we have
\begin{equation}\label{eq:Tq}
\widetilde T(z_1, \bar x)=\prod_{v\in \widetilde V}\widetilde P_v(z_1,\bar x).
\end{equation}
For $j=1,\dots, p$, denote  $a_j=\balpha_{x_j}$,
  and let $N_j$ be the size of $\Phi_{x_j}=(C_{x_j}, D_{x_j})$.
We assert that
\begin{equation}\label{eq:tauv}
x^{\mu(v)}
=(z_1,\bar x)^{\widetilde \mu(v)} \circ \left (z_1\mapsto \prod_{j=1}^p x_j^{a_jN_j}\right ),
~~\text{ for }v\in \widetilde V.
\end{equation}
Notice that \eqref{eq:tauv} holds if the following  holds:
\begin{equation}\label{eq:compose}
 \mu(v,x_j)=\left \{
 \begin{array}{ll}
 \widetilde \mu(v,z_1) a_jN_j, & \text{ if }  1\leq j\leq p;\\
 \widetilde \mu(v,x_j) & \text{ if } p+1\leq j\leq n.
 \end{array}
 \right .
 \end{equation}
If $v$ is  an ancestor of $z_1$ or $v=z_1$, we  have
$\mu(v,x_j)=\widetilde \mu(v,z_1) \mu(z_1,x_j)=\widetilde \mu(v,z_1) a_jN_j,$ $1\leq j\leq p;$
if   $v$ is not an ancestor of $z_1$, we have
$\mu(v,x_1)=\cdots =\mu(v,x_p)=\widetilde \mu(v,z_1)=0.$
So the first assertion of  \eqref{eq:compose}  holds.
The second assertion   holds since the path from $v$ to $x_j$ in $\Gamma$  is also a path in $\widetilde \Gamma$. This proves \eqref{eq:compose} and \eqref{eq:tauv} follows.

Now, we consider the relations between $P_v$ and $\widetilde P_v$ for $v\in \widetilde V$.

If $v=z_1$, since $z_1$ is a top of $\widetilde V$, we have $ \widetilde P_{z_1}(z_1,\bar x)= C_{z_1}(z_1)$;
moreover, $\mu(z_1)=(a_1N_1,\dots, a_pN_p, 0,\cdots, 0)$ and $\mu(x_j)=\be_j$, $1\leq j\leq n$. Denote $\ba=(a_1,\dots,a_p)$, we have
\begin{equation}\label{eq:9.10}
\begin{array}{rl}
P_{z_1}(x) &=C_{z_1}\left(\prod_{j=1}^p x_j^{a_jN_j}\right)L^{\bdelta_{z_1}}_\ba(x_1^{N_1},\dots, x_p^{N_p})\\
       &=\left (\widetilde P_{z_1}(z_1,\bar x)\circ(z_1\mapsto \prod_{j=1}^p x_j^{a_jN_j})\right ) L^{\bdelta_{z_1}}_\ba(x_1^{N_1}, \cdots,  x_p^{N_p}).
\end{array}
\end{equation}

 We claim that
\begin{equation}\label{eq:PV}
P_v(x)
=\widetilde P_v(z_1,\bar x)\circ \left (z_1\mapsto \prod_{j=1}^p x_j^{a_jN_j}\right ), \quad \text{ if }v\in \widetilde V \setminus\{z_1\}.
\end{equation}

If $v\in \{x_{p+1},\dots,x_n\}$, we have $P_v(x)=\widetilde P_v(z_1,\bar x)=C_v(v)$, clearly \eqref{eq:PV} holds.

If $v\in \widetilde V \setminus\{z_1,x_{p+1},\dots,x_n\} $,
denote the offsprings of $v$ by $y_1,\cdots,y_m$.
 Set $\bb=(\balpha_{y_1},\cdots,\balpha_{y_m})$, and let $M_i$ be the  size of $\Phi_{y_i}$ for $1\leq i\leq m$.
Since $C_v=\widetilde C_v$ and $\delta_v=\widetilde \delta_v$, we have
$$
\begin{array}{l}
P_v(x)=C_v(x^{\mu(v)}) L^{\bdelta_v}_{\bb}(x^{M_1\mu(y_1)},\dots, x^{M_m\mu(y_m)}),\\
\widetilde P_v(z_1,\bar x) =C_v((z_1,\bar x)^{\widetilde \mu(v)})L^{\bdelta_v}_\bb((z_1,\bar x)^{M_1\widetilde \mu(y_1)},\dots, (z_1,\bar x)^{M_m\widetilde \mu(y_m)}).
\end{array}
$$
Notice that \eqref{eq:tauv} holds for  all $v, y_1,\dots, y_m\in \widetilde V\setminus\{z_1\}$, so \eqref{eq:PV}   follows.  The claim is proved. Therefore,  we have
$$
\begin{array}{rll}
& T(x)
=\displaystyle \left (\prod_{j=1}^p C_{x_j}(x_j) \right ) L^{\bdelta_{z_1}}_\ba(x_1^{N_1},\dots, x_p^{N_p})\cdot \widetilde T\left(\prod_{j=1}^p x_j^{a_jN_j},\bar x\right ) \quad &\text{(By \eqref{eq:treee})}\\
=& \displaystyle\left (\prod_{j=1}^p P_{x_j}(x)\right ) L^{\bdelta_{z_1}}_\ba(x_1^{N_1},\dots, x_p^{N_p})\cdot \prod_{v\in \widetilde V}\widetilde P_v(z_1,\bar x)\circ (z_1\mapsto \prod_{j=1}^p x_j^{a_jN_j}) \quad &\text{(By \eqref{eq:Tq})}\\
=&\displaystyle\left (\prod_{j=1}^p P_{x_j}(x)\right ) \left (L^{\bdelta_{z_1}}_\ba(x_1^{N_1},\dots, x_p^{N_p}) \widetilde P_{z_1} (\prod_{j=1}^p x_j^{a_jN_j}, \bar x)\right ) \left (\prod_{v\in \widetilde V\setminus \{z_1\}}   P_v(x)  \right )\quad &\text{(By \eqref{eq:PV})} \\
=&\displaystyle\prod_{v\in V} P_v(x). \quad &\text{(By \eqref{eq:9.10}.)}
\end{array}
$$
By symmetry, we have $\SJ(x)=\prod_{v\in V} Q_v(x)$.
 The theorem is proved.
\end{proof}

\section{\textbf{Proof of Theorem \ref{thm:main-2} and an example}}

\begin{proof}[\textbf{Proof of Theorem \ref{thm:main-2}}]
Thanks to Theorem \ref{thm:separable}, we only need show that any primitive $\N^*$-pair $(T, \SJ)$ can be generated by a weighted tree. If  $(T,\SJ)$ is an $\N$-pair,  then it is generated by the tree $\Lambda_0=(\{\phi\}, \emptyset)$
  with initial pair  $(T_0,\SJ_0)=(T, \SJ)$. 

 Suppose  $(T, \SJ)$  is generated by a tree $(V, \Gamma)$ with weight
$((T_0, \SJ_0), \bdelta, \balpha, \Phi)$. In the following, we show that    any one step  extension $(T^*, \SJ^*)$ of it can also be generated by a weighted tree. 
 Denote the tops of $(V, \Gamma)$ by $x_1,\dots, x_n$, and  assume   the extension is  along the variable $x_1$.
Denote  $\bar x=(x_2,\dots, x_n)$.

\textit{Case 1.} Suppose $(T^*, \SJ^*)$ is a first type extension of $(T, \SJ)$ given by
$$
T^*(x)=C(x_1)T(x_1^N, \bar x), \quad  \SJ^*(x)=D(x_1)\SJ(x_1^N, \bar x),
$$
where $(C,D)$ is a $\{0,\dots,N-1\}$-pair. Denote $\Phi_{x_1}=(C_{x_1},D_{x_1}).$
We define a new weight of $(V, \Gamma)$ by only modifying
  $\Phi_{x_1}$ to   the interval pair $(C+N C_{x_1},D+N D_{x_1})$.
 Then  $(T^*, \SJ^*)$  is generated by the tree
$(V, \Gamma)$ with this new weight.

\textit{Case 2.} Suppose  $(T^*, \SJ^*)$ is a second type extension of $(T, \SJ)$ given by
$$
T^*(y, \bar x)=L^s_\bb(y) T(y^{\bb}, \bar x),
\quad
\SJ^*(y, \bar x)=L^{1-s}_\bb(y) \SJ(y^{\bb},  \bar x),
$$
where $y=(y_1,\dots, y_m)$, $\bb=(b_1,\dots, b_m)>0$ and $s\in \{0,1\}$.
We set $V'=V\cup\{y_1,\dots, y_m\}$, and let $\Gamma'=\Gamma\cup\{[x_1,y_j];~ j=1,\dots, m\}$ be the edge set.
We define a weight $((T_0, \SJ_0), \bdelta', \balpha', \Phi')$ of $(V',\Gamma')$, which is the extension of  the original weight by adding the following parameters:
$(i)$ Set $\bdelta'_{x_1}=s$; \ $(ii)$ For $j\in \{1,\dots, m\}$, set $\balpha'_{y_j}=b_j$ and $\Phi'_{y_j}=(\{\bzero\}, \{\bzero\})$.
Then  $(T^*, \SJ^*)$ is  generated by  $(V', \Gamma')$ with the above weight.
\end{proof}

   \begin{figure}
  \begin{center}
  \includegraphics[height=0.21\textheight]{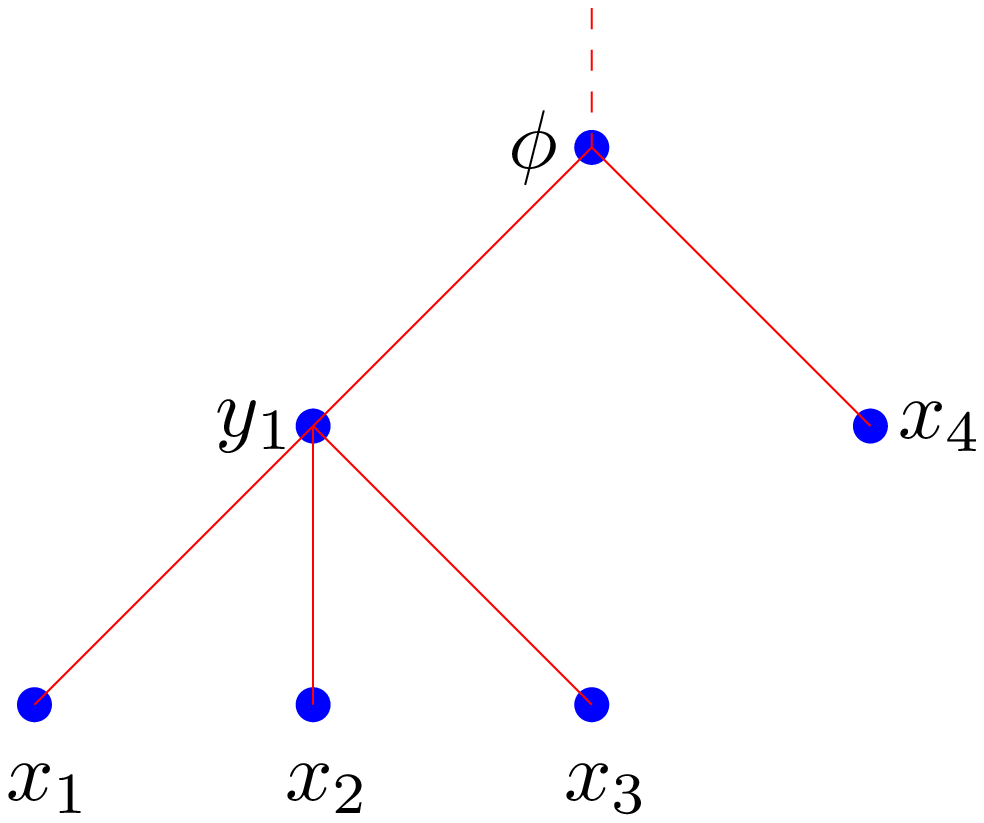}\\
  \end{center}
  \vskip -0.4cm
  \caption{A rooted tree}
\label{fig:tree}
\end{figure}

\begin{table}[h]
 The function $\balpha$ and $\Phi$:\\
 \smallskip
\centering
\begin{tabular}{|c|c|c|c|c|c|c|}
\hline
node &  $y_1$         & $x_4$         &$x_1$        & $x_2$      &$x_3$        \\\hline
$\balpha$  &     2          & 3             &    1        &    1       &    1       \\\hline
$C$ &  $\{0,2\}$      &  $\{0,2\}$      &    \{0\}       &     \{0\}       &     \{0\}       \\\hline
$D$ & $\{0,1\}$&$\{0,1\}$&\{0\} &\{0\} & \{0\}  \\\hline
$N$  &      4          &    4  &1 & 1&1\\\hline
\end{tabular}
\smallskip
\caption{We set $\bdelta$ defined on $\{\phi,y_1\}$ to be constantly zero. }\label{tab:tab1}
\end{table}

\begin{example}\label{example:2}{\rm
 Let $(\Gamma, V)$ be a rooted tree
indicated by Figure \ref{fig:tree}. The node set
 $V=\{\phi, y_1, x_1, x_2, x_3, x_4\}$. We set the initial $\N$-pair to be
 $$T_0(\phi)=1+\phi^2, \quad \SJ_0(\phi)=(1+\phi)\sum_{k=0}^\infty \phi^{4k};$$
 the other parameters of the weight is given by Table \ref{tab:tab1}.
Set
$$
V_0=\{\phi\}, \ V_1=\{\phi; y_1, x_4\}, V_2=\{\phi; y_1, x_4; x_1, x_2, x_3\}.
$$
Let $(V_j,\Gamma_j),j=0,1,2,$ be  subtrees of $(V,\Gamma)$.
The $\N^*$-tile $T_1$ corresponding to the subtree $(V_1,\Gamma_1)$ is
$
T_1(y_1,x_4)=(1+y_1^2)(1+x_4^2)(1+y_1^{16}x_4^{24}).
$
The $\N^*$-tile $T_2$ corresponding to the tree $(V_2,\Gamma_2)=(V,\Gamma)$ is
$$
\begin{array}{rl}
T_2(x_1,x_2,x_3,x_4)
=&1\cdot 1\cdot 1\cdot (1+(x_1x_2x_3)^2)\cdot (1+x_4^2)\cdot\left (1+(x_1x_2x_3)^{16}x_4^{24} \right ).\\
=&P_{x_1}(x)\cdot P_{x_2}(x) \cdot P_{x_3}(x)\cdot P_{y_1}(x)\cdot P_{x_4}(x)\cdot P_{\phi}(x).
\end{array}
$$
 }\end{example}


\begin{thebibliography}{99}
\bibliographystyle{ieee}

\bibitem{CM99} E. M. Coven and A. Meyercovitz, \textit{Tiling the integers with translates of one finite set,}
J. Algebra, \textbf{212} (1999), 161--174.

\bibitem{deB50} N. G. de Bruijn, \textit{On bases for the set of integers}, \emph{Publ. Math.,} \textbf{1}(3)(1956), 583--590.

\bibitem{deB56} N. G. de Bruijn, \textit{On number systems}, \emph{Nieuw Arch. Wisk.,} \textbf{4}(3)(1956), 15--17.

\bibitem{Eigen} S. Eigen and A. Hajian, Sequences of integers and ergodic transformations, \textit{Adv. Math.,}
\textbf{73} (1989), 256--262.

\bibitem{Hansen} R. T. Hansen, \textit{Complementing pairs of subsets in the plane}, Duke Math. J. \textbf{36}(1969), 441--449.

\bibitem{Haj} G. Haj\'os, Sur la factorisation des groupes ab\'elians, \~Casopis P\~est. math. Fys. (3) \textbf{4} (1950), 157--162.

\bibitem{Long} C. T. Long, Addition theorems for sets of integers, \textit{Pacific J. Math.} \textbf{23}(1967), 107--112.

\bibitem{Nath72} M. B. Nathanson, \textit{Complementing sets of $n$-tuples of integers}. Proc. Amer. Math. Soc. \textbf{34} (1972), 71-72.

\bibitem{Newman} D. J. Newman, Tesselation of integers. \textit{J. Number Theory} \textbf{9} (1977), 107-111.

\bibitem{Niven} I. Niven, \textit{A characterization of complementing sets of pairs of integers.} Duke Math. J. \textbf{38} (1971), 193--203.


\bibitem{Sands} A. D.  Sands, \textit{On Keller's conjecture for certain cyclic groups}, Proc. Edinburg Math. Soc. (1), \textbf{22} (1979), 17--21.

\bibitem{Szegedy} Szegedy, M. Algorithms to tile the infinite grid with finite clusters. Proceedings of the 39th Annual Symposium on the Foundations of Computer Science, FOCS ¡¯98. November8¨C111998, Palo Alto, CA. pp.137¨C145. IEEE Computer Society.


\bibitem{Swenson} C. Swenson, Direct sum subset decompositions of $\Z$. \emph{Pacific J. Math.} \textbf{53} (1974), 629--633.

\bibitem{Szabo} S. Szab\'o, A type of factorization of finite Abelian groups, \textit{Discrete Math.} \textbf{54}(1985), 121--124.

\bibitem{Tij} R. Tijdeman, \textit{Decomposition of the integer as a direc sum of two subsets}, in "Number Theory(Paris, 1992-1993)," London Math. Soc. Lecture Note Ser.,Vol.215, 261--276. Cambridge Univ. Press, 1995.

\bibitem{Vaidya} A. M. Vaidya, \textit{On complementing sets of nonnegative integers}, Math. Mag. \textbf{39}(1966), 43--44.


\end{thebibliography}
\end{document}